\documentclass[letterpaper,12pt]{article}
\usepackage{amsmath, amsfonts,amssymb,latexsym,amsthm}
\usepackage{fullpage}
\usepackage{parskip}
\usepackage{graphicx}
\usepackage{mathtools}
\usepackage{wasysym}
\usepackage{caption}

\newtheoremstyle{mystyle}
  {\topsep} % Space above
  {\topsep} % Space below
  {} % Body font
  {} % Indent amount
  {\bfseries} % Theorem head font
  {.} % Punctuation after theorem head
  {.5em} % Space after theorem head
  {} % Theorem head spec (can be left empty, meaning `normal')
	
\theoremstyle{mystyle}
\newtheorem{theorem}{Theorem}[section]
\newtheorem{definition}{Definition}[section]

\newtheorem{corollary}{Corollary}[section]
\newtheorem{conjecture}{Conjecture}[section]
\newtheorem{question}{Question}[section]
\newtheorem{lemma}{Lemma}[section]

\begin{document}
\title{Dynamics of piecewise translation maps}
\author{Sang Truong \\ University of California, Irvine}
\date{\today}

\maketitle{}
\bigskip

%%%%%%%% ABSTRACT %%%%%%%%%%%%%%%%%%%%%%%%%%%%%%%%%%%%%
\begin{abstract}
In this paper, we introduce a generalized piecewise translation map on the Euclidean space. We provide a special case when this map is always of finite type. For a finite type map in this case, we form conjectures on the semi-continuity of the attractor. Moreover, we also provide some conjectures and pictures on piecewise translations on a two dimensional torus.                  
\end{abstract}
%%%%%%%%%%%%%%%%%%%%%%%%%%%%%%%%%%%%%%%%%%%%%%%%%%%%%%%

\section{Background}
The goal of this paper is to study the dynamics of piecewise translation maps. These maps are special types of piecewise isometries on a region $\Omega$ in the Euclidean space. To define piecewise isometries, we partition $\Omega$ into smaller regions $P_0, \cdots, P_m$ such that these pieces intersect only at their boundaries. We also assume that the boundary of each piece has zero Lebesgue measure. A map $f: \Omega \to \Omega$ is called a \emph{piecewise isometry} if the restriction of $f$ to any piece $P_i$ is an isometry, i.e. it preserves distance on each piece. Piecewise isometries, especially piecewise translation maps, serve as models for many engineering applications, such as the second order digital filter \cite{CL, D}, the sigma-delta modulator \cite{ADF, D1, F}, the buck converter \cite{D, D2}, the three-capacitance model \cite{SAO}, error diffusion in digital printing \cite{AKMPST, Adler2012}, and machine learning.\\

One of the simplest examples of such piecewise translation can be formed using \textit{interval exchange transformations} (IET). These transformations are a generalization of a classical dynamical system -- circle rotation. To define an IET, let $\Omega = [0,1) \subseteq \mathbb{R}$ be partitioned into a finite number of distinct intervals. The map $f:\Omega \to \Omega$ that permutes the intervals defines a piecewise isometry. 

 \begin{figure}
	\centering
		\includegraphics{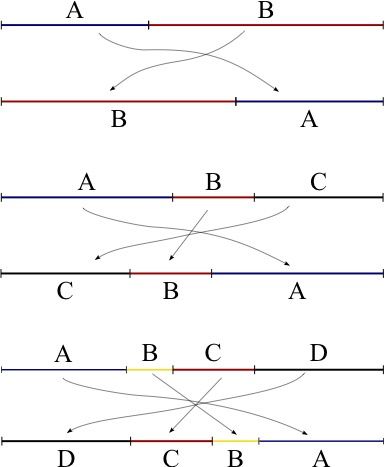} \caption{Interval Exchange Transformation (IET)\label{f.1}}
		\label{fig:iet_example}
	\end{figure}

Several nice properties of IET have been discovered. In particular, M. Kean \cite{K} showed that every orbit of $f$ is dense in $[0,1)$. He also proved that if the partition has only two or three intervals, then all minimal IETs are also uniquely ergodic. The statement is not necessarily true if the number of intervals exceeds 3. However, it was proven by Masur \cite{M} and Veech \cite{Ve} that almost every interval exchange transformation is uniquely ergodic. This motivates a generalized conjecture that most piecewise isometries (restriction to the attractors) are also uniquely ergodic.

We can also consider a variation of interval exchange transformations, namely \textit{interval translation mappings} (ITM). These mappings are introduced by M. Boshernitzan and I. Kornfeld in 1995 \cite{BK}. For ITM, the images of the intervals are not required to be disjoint. We classify ITM into two types: finite and infinite. In the case $f(\Omega) \neq \Omega$, consider the sequence $\Omega_0 = \Omega, \; \Omega_1 = f(\Omega_0), \ldots$. From this, we get the nested chain $\Omega_0 \supseteq \Omega_1 \supseteq \ldots$. If this chain stabilizes we say the map $f: \Omega \to \Omega$ is of finite type; otherwise, $f$ is of infinite type. We see that the map can only be of infinite type if there are more than two elements in the partition. In particular, Boshernitzan and Kornfield constructed an ITM of infinite type with three branches \cite{BK}.

Another class of one dimensional piecewise isometries is the double rotations of the circle. Let $\Omega = S^1=[0,1]/_{\{0\,\sim\,1\}}$ be partitioned into two arcs $P_0$ and $P_1$. We define the double rotation map $f:S^1 \to S^1$ as following:
$$
f(x)=\left\{
       \begin{array}{ll}
         x+\alpha_0 \;(\text{mod}\ 1), & \hbox{if $x\in P_0$;} \\
         x+\alpha_1 \;(\text{mod}\ 1), & \hbox{if $x\in P_1$.}
       \end{array}
     \right.
$$

Almost all double rotations of the circle are uniquely ergodic, even most of those of infinite type \cite{BC}. We would like to emphasize that this text is the result of an undergraduate research project under the guidance of Prof. Gorodetski, based on the original ideas and results from \cite{GVW} and \cite{V2}. The pictures were inspired by those from \cite{GVW} and \cite{V2}, but were recreated by independent reproduction of their numerical experiments. \\

\section{General set-up}
Let $\Omega$ be a compact set in the Euclidean space $\mathbb{R}^d$. Suppose $\Omega = \bigcup_{i=0}^{m-1}B_i$ where $B_i$'s are compact for all $i = 0, 1, \ldots, m-1$. Here we allow the subsets $B_i$'s to have non-trivial intersections. We define a translation map on each $B_i$'s by $T_i(x) = x + v_i$, where $v_i$ is a vector in $\mathbb{R}^d$. Let $\mathcal{X}$ be the space of all compact subsets of $\Omega$. \\

\begin{definition}
The map $F : \mathcal{X} \to \mathcal{X}, \,\, F(K) = \bigcup_{i=0}^{m-1}T_i(K \cap B_i)$ is called a \emph{piecewise translation of m branches} on $\Omega$.
\end{definition}

 \begin{figure}
	\centering
		\includegraphics[scale=3]{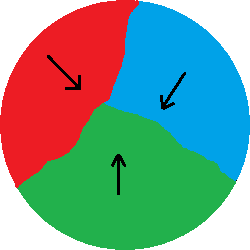}
		\caption{Piecewise translation of 3 branches on a 2-dimensional disk.}
		\label{fig:disk_example}
	\end{figure}
	
Each $B_i$ is called a branch of $\Omega$. Locally, if $x$ is in both $B_i$ and $B_j$, then $x$ will be translated by both $v_i$ and $v_j$. Figure \ref{fig:disk_example} represents a piecewise translation of 3 branches on a 2-dimensional disk. These three branches intersect only at their boundaries, which have zero Lebesgue measure. Figure \ref{fig:disk_example} is a simplified model of machine learning. In order for the machine to learn, we need multiple inputs from the environment. Each input is represented by a translation vector. The process of machine learning is modeled by the dynamics of applying the piecewise translation map many times. Therefore, it is natural to define what we get after applying the map over and over again, which we will call the \emph{attractor}.\\
\begin{definition}
	The set $A = \bigcap_{n=1}^{\infty}F^n(\Omega)$ is called the \emph{attractor} of the map $F$. If $A = F^N(\Omega)$ for some $N \in \mathbb{N}$, then we say the map $F$ (or the attractor $A$) is of \emph{finite type}. Otherwise, $F$ is of \emph{infinite type}.
\end{definition}
     
Because $\{F^n(\Omega)\}_{n=1}^{\infty}$ is a non-increasing sequence of non-empty compact sets, it has a non-empty intersection. Therefore, the attractor is always non-empty and compact. Moreover, the attractor should be invariant under the map $F$. This expected property is a direct consequence of the following lemma.\\

\begin{lemma} \label{invariant_attractor_lemma}
Let $K_0 \supseteq K_1 \supseteq \ldots$ is a non-increasing sequence of non-empty compact sets in $\mathcal{X}$, then $F\left(\bigcap_{n=0}^{\infty}K_n\right) = \bigcap_{n=0}^{\infty}F(K_n)$.
\end{lemma}   

\begin{proof}
	The inclusion $F\left(\bigcap_{n=0}^{\infty}K_n\right) \subseteq \bigcap_{n=0}^{\infty}F(K_n)$ is trivial. In order to show the other direction, let $x \in \bigcap_{n=0}^{\infty}F(K_n)$, then $x \in F(K_n)$ for all $n \in \mathbb{N}$. Let $H = \{y \in \Omega \mid F(\{y\}) \ni x\}$. Note that $H$ is a finite union of compact sets, hence is compact. Indeed, $H$ has finitely many points in $\Omega$. Since $x \in F(K_n)$ for all $n$, the sets $H_n = H \cap K_n$ are non-empty. So $\{H_n\}_{n=0}^{\infty}$ is a non-increasing sequence of non-empty compact sets. Therefore it has a non-empty intersection; i.e. $\exists y \in H$ such that $y \in \bigcap_{n=0}^{\infty}K_n$. As a consequence, $\{x\} \subseteq F(\{y\}) \subseteq F\left(\bigcap_{n=0}^{\infty}K_n\right)$. This shows that $x \in F\left(\bigcap_{n=0}^{\infty}K_n\right)$.\\
\end{proof}

\begin{corollary} \label{invariant_attractor_cor}
	The attractor is invariant under the map $F$; i.e. $F(A) = A$.\\
\end{corollary}

In some cases, a small change in the translation vectors $v_0, \,v_1, \ldots, \,v_{m-1}$ leads to a huge change in the geometry of $A$. In other words, $A$ is not ``continuous'' in the parameter space. It is convenient to introduce the Hausdorff metric on $\mathcal{X}$ to study the continuity of $A$. Here we denote the $\varepsilon$-neighborhood of a set $X$ with the standard Euclidean metric by $X_{\varepsilon}$. We will show that the orbit of $\Omega$ gets arbitrarily close to the attractor with respect to the Hausdorff metric.\\
\begin{definition}
	Let $X, Y \in \mathcal{X}$, the Hausdorff distance between $X$ and $Y$ is defined by $$d_H(X, Y) = \inf\{\varepsilon > 0 \mid \mbox{$X \subseteq Y_{\varepsilon}$ and $Y \subseteq X_{\varepsilon}$}\}.$$
\end{definition}

\begin{lemma} \label{A_epsilon}
	Let $K_0 \supseteq K_1 \supseteq \ldots$ be a non-increasing sequence of non-empty compact sets in $\mathcal{X}$. For every $\varepsilon > 0$, there exists $N \in \mathbb{N}$ such that $K_N \subseteq \left(\bigcap_{n=0}^{\infty}K_n\right)_{\varepsilon}$.
\end{lemma}

\begin{proof}
	Denote $S_n = K_n \setminus \left(\bigcap_{n=0}^{\infty}K_n\right)_{\varepsilon}$, then $S_n$ are compact and $S_0 \supseteq S_1 \supseteq \ldots$ If all $S_n$ are non-empty, then they have a non-empty intersection. This leads to $\left(\bigcap_{n=0}^{\infty}K_n\right) \setminus \left(\bigcap_{n=0}^{\infty}K_n\right)_{\varepsilon} \not= \varnothing$ (contradiction \lightning). Therefore, $\exists N \in \mathbb{N}$ such that $S_N = \varnothing$, which means $K_N \subseteq \left(\bigcap_{n=0}^{\infty}K_n\right)_{\varepsilon}$.\\ 
\end{proof}

\begin{corollary} \label{orbit_of_Omega}
	For every $\varepsilon > 0$, there exists $N \in \mathbb{N}$ such that $d_H(A, F^n(\Omega)) < \varepsilon$ for all $n \geq N$.
\end{corollary}

\section{Finiteness theorem when $m = d+1$}

In this section, we are interested in piecewise translation maps of $m = d+1$ branches in $\mathbb{R}^d$. We assume that $\text{rank} \{v_0,v_1,\ldots,v_{m-1}\} = d$; that is, the vectors in $\mathcal{B} = \{v_1 - v_0, v_2 - v_0, \ldots, v_{m-1} - v_0\}$ form a basis for $\mathbb{R}^d$. Let $\mathcal{L}$ be the lattice generated by $\mathcal{B}$, then $\mathbb{T} = \mathbb{R}^d / \mathcal{L}$ is a d-dimensional torus. In this case, $v_0$ serves as the translation vector in $\mathbb{T}$; i.e. we have the torus rotation map $R: \mathbb{T} \to \mathbb{T}, \, R(x + \mathcal{L}) = x + v_0 + \mathcal{L}$. Atom-wise, the action of $F$ is the same as the action of $R$ on $\mathbb{T}$. Globally, it is also true and can be seen through the canonical projection $\pi: \mathbb{R}^d \to \mathbb{T}, \, \pi(x) = x + \mathcal{L}$. \\

\begin{lemma}\label{torus_map_semi_lemma}
	If $K \subseteq \Omega$ is compact, then $\pi \circ F(K) = R \circ \pi(K)$.
\end{lemma}              
\begin{proof}
	Let $\phi \in \pi \circ T(K)$, then $\phi \in \bigcup_{i=1}^{m-1}\pi(T_i(B_i \cap K))$. So $\phi = \pi(T_j(B_j \cap K))$ for some $j$. Therefore, $\exists y \in B_j \cap K$ such that $\phi = \pi(y + v_j)$. As a result, $$\phi = y + v_j + \mathcal{L} = y + v_0 + (v_j - v_0) + \mathcal{L} = y + v_0 + \mathcal{L} = R \circ \pi(y) \in R \circ \pi(K).$$
	For the backward inclusion, let $\phi \in R \circ \pi(K)$, then $\phi = R \circ \pi(y) = y + v_0 + \mathcal{L}$ for some $y \in K$. Without loss of generality, we can assume $y \in B_j$. In this case, $\phi = \pi(y + v_j) \in \pi(T_j(B_j \cap K))$. Hence, $x \in \pi \circ F(K)$.\\
\end{proof}

It was proven that any piecewise translation of an interval with 2 branches is of finite type \cite{V}. However, this does not hold for dimension 2 or higher. Indeed, not all double rotation on a torus is of finite type. In the case where $m=d+1$, there is a space of parameters (translation vectors) where we will get finite type attractors. We assume that the translation vectors $v_0, \,v_1, \ldots, v_{m-1}$ are \emph{rationally independent}, i.e. they are linearly independent in $\mathbb{Q}$-vector space. This is equivalent for the toral map $R$ to be \emph{ergodic}. In this case, we say that the translation vectors are \emph{minimal}. With this assumption, the map $F$ is necessarily of finite type. We start with some definitions that will be used in the proof.\\

\begin{definition}
A set $X$ is said to be \emph{invariant} under the map $f:M\to M$ if $f(X) = X$.\\
\end{definition}

\begin{definition}
	A toral map $R: \mathbb{T} \to \mathbb{T}$ is said to be \emph{minimal} if every orbit is dense; i.e there is no proper closed invariant subset.\\
\end{definition}

\begin{definition}
	A set $X \subseteq \mathbb{R}^d$ is \emph{nowhere dense} if $\text{int } \bar{A} = \varnothing$.\\ 
\end{definition}

\begin{theorem}\label{ergodic_finite_theorem}
	Piecewise translation maps with \emph{rationally independent} translation vectors are of finite type. 
\end{theorem}
\begin{proof}
	The idea is to show that $F^N(\Omega)$ is contained in some $\varepsilon$-neighborhood of $A$ but not contained in $\varepsilon$-neighborhood of $\partial A$ for some $N>0$. In this case, $F^N(\Omega)$ must be in the interior of $A$, thus a subset of $A$.\\\\
	Since $A$ is compact, non-empty and $\pi$ is continuous, $\pi(A)$ is compact and non-empty. Moreover, $R \circ \pi(A) = \pi \circ F(A) = \pi(A)$ by Lemma \ref{invariant_attractor_lemma}, so $\pi(A)$ is invariant under $R$. We know that ergodic torus rotation is minimal; therefore, there is no proper closed invariant sets. Since $\pi(A)$ is non-empty, $\pi(A) = \mathbb{T}$. Since $A$ is closed, $\partial A$ is nowhere dense in $\mathbb{R}^d$. Since $\Omega$ is compact, the projection $\pi$ is at most finitely branched, so locally $\pi(\partial A)$ is a finite union of nowhere dense sets, and thus is nowhere dense in $\mathbb{T}$. Therefore, we can take a small ball $B$ in the complement of $\pi(\partial A)$ and $\varepsilon > 0$ such that for $U = (\pi(\partial A))_{\varepsilon}$ we have $U \cap B = \varnothing.$\\\\
	By Lemma \ref{A_epsilon}, there exists $N > 0$ such that $F^n(\Omega) \subseteq A_{\varepsilon}$ for all $n > N$. So we only need to show that $F^n(\Omega)$ is not in $(\partial A)_{\varepsilon}$ if $n$ is large enough. It is known that every ergodic toral rotation map is uniformly minimal; i.e. for the open set $B$, there exists $M > 0$ such that $\forall \phi \in \mathbb{T}, \, R^{k(\phi)} \phi \in B$ for some $0 < k(\phi) < M$. Note that $k$ depends on $\phi$ while $M$ does not.\\\\
	\underline{Claim}: $\pi \circ F^{n+k}(\{x\}) \cap U = \varnothing$ for some $0<k<M$.\\
	Indeed, we can apply Lemma \ref{torus_map_semi_lemma} and the uniformly minimal property for $\phi = R^n \circ \pi(x)$:
	\begin{equation}
		\begin{aligned}
		\pi \circ F^{n+k}(\{x\}) &= R^{n+k}\circ \pi(\{x\}) 
														\\&= \{R^{n+k}\circ \pi(x)\}
														\\&= \{R^k(\underbrace{R^n\circ\pi(x)}_{\phi})\} \subseteq B.
		\end{aligned}
	\end{equation}
	Since $\pi$ is continuous, we can choose a positive $\delta < \varepsilon$ such that $\pi((\partial A)_{\delta}) \subseteq \left(\pi(\partial A)\right)_{\varepsilon} = U$. Hence $\pi \circ F^{n+k}(\{x\}) \cap \,U = \varnothing.$ Therefore, $F^{n+k}(\{x\}) \,\cap\, (\partial A)_{\epsilon} = \varnothing$. This leads to $F^{n+k}(\{x\}) \subseteq \text{int } A \subseteq A$. Since $k = k(x)$ is uniformly bounded by $M$, we have $F^{N+M}(\Omega) \subseteq A$, i.e. the map $F$ is of finite type.\\    
\end{proof}

 \begin{figure}
	\centering
		\includegraphics[scale=2]{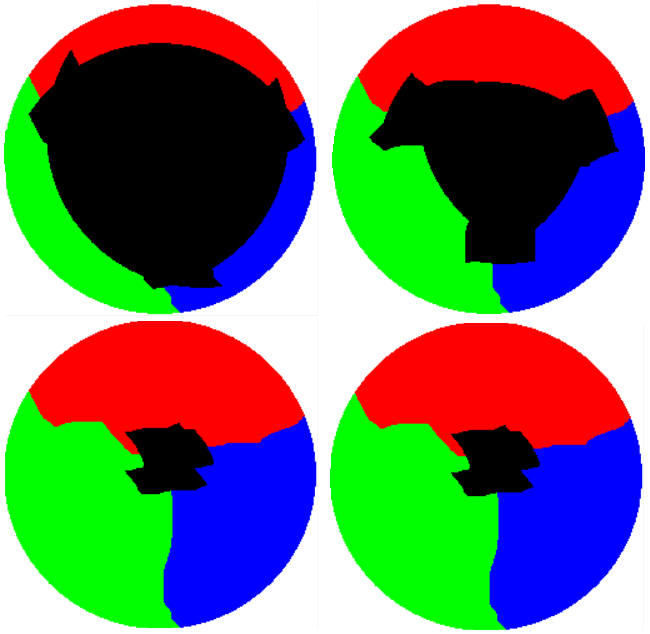}
		\captionsetup{justification=centering}
		\captionsetup{format=hang}
		\caption{The black regions represent $F^n(\mathbb{D}^2)$ for $n=1,2,8,9$.\\The disk converges to the attractor after 8 iterations. The map is of finite type.}
		\label{fig:disk_finite_type}
	\end{figure}

\section{Semi-continuity of the attractors}
In general, it is not true that the attractor is continuous with respect to Hausdorff metric in the parameter space. Even in the case $m = d+1$, we observed that sometimes the two separate pieces of the attractor reglue to each other. One of our concerns has not been proven yet.\\ 

\begin{question}
	Is it true that for a generic family of partitions, the attractor depends continuously on the translation vectors? \\
\end{question}   

However, in the case $m=d+1$, the question of semi-continuity of the attractor is easier to answer. We will show that for a particular family of partitions, finite type attractors are semi-continuous on the translation vectors. Before showing the proof for this, we need a formal definition for semi-continuity of the attractors.\\

\begin{definition}
	Let $\mathcal{W}$ be the space of translation vectors such that the associated piecewise translation map $F_{v_0, v_1, \ldots, v_{m-1}}$ is of finite type with a fixed partition $\mathcal{B}$. Denote $X_a$ the attractor associated with the parameter $a \in \mathcal{W}$. We say that the attractors are \emph{semi-continuous} on $\mathcal{W}$ if for every $a \in \mathcal{W}$, there exists a neighborhood $O(a)$ such that for each $b \in O(a)$, we have $X_b \subseteq (X_a)_{\varepsilon}$.  \\
\end{definition}  

Note that the semi-continuity of the attractor corresponds to the \emph{pseudo}-Hausdorff metric $d(X, Y) = \sup\{d(x, Y) \mid x \in X\}$. It is clear that $$d_H(X, Y) = \max \{d(X, Y), \, d(Y,X)\}.$$

\begin{conjecture} \label{semi_continuity}
 If $m = d+1$, there exists a generic parameter value such that the attractors are semi-continuous in $\mathcal{W}$.
\end{conjecture}

\section{Piecewise translation maps on tori}
In this section, we discuss a special class of piecewise translations of two branches on a two dimensional torus called \emph{double rotations}. Double rotations were discussed in \cite{SIA} and \cite{G}. Synchronization properties of random double rotations on tori were studied in \cite{GK}. Here we will show some pictures of double rotations with finite and infinite types. \\

\begin{definition}
	Let $\mathbb{T}$ be a two-dimensional torus and $v_0, v_1$ are vectors in $\mathbb{R}^2$. Consider a subset $R$ of $\mathbb{T}$. The map $f: \mathbb{T} \to \mathbb{T}$ defined as 
	$$
f(x)=\left\{
       \begin{array}{ll}
         x + v_0 \;(\text{mod}\ 1), & \hbox{if $x\in R$;} \\
         x + v_1 \;(\text{mod}\ 1), & \hbox{otherwise}
       \end{array}
     \right.
$$
is called a \emph{double rotation} on $\mathbb{T}$. \\
\end{definition}

In Figure \ref{fig:torus_rotation}, the torus is partitioned into two regions - green and blue. We assign a different translation vector to each region. After applying the map, we see that a region from the torus is lost (the red region). In this case, it is the translational image of the green region. Applying the map one more time, we get an additional lost region. Again, this is a translational image of the previous red region. After 500 times, we no longer get any additional lost region. The map is of finite type and its attractor is the blue part in Figure \ref{fig:finite_attractor_torus}. On the other hands, in Figure \ref{fig:infinite_attractor_torus}, we do not get the attractor even after 5000 iterates. The map seems to be of infinite type.\\  

Let $\mathcal{V}$ be the space of parameters $\{v_0, v_1\}$ with a fixed partition such that the double rotation map $f_{v_0, v_1}$ is of finite type. By computer simulation, we conjecture that the map is of finite type for almost every parameter. If the map is finite, we also expect the attractor is semi-continuous with respect to Hausdorff distance. This behavior has already been seen in the previous section.\\  
\begin{conjecture}
$\mathcal{V}$ is an open, dense, and full measure subset of the set of all parameters. \\
\end{conjecture}

\begin{conjecture}
The attractor is semi-continuous with respect to the parameters in $\mathcal{V}$.\\
\end{conjecture}

\begin{question}
Is it true that the attractor is continuous with respect to the parameters in $\mathcal{V}$?\\
\end{question}

\begin{figure}
	\centering
		\includegraphics[scale=1.7]{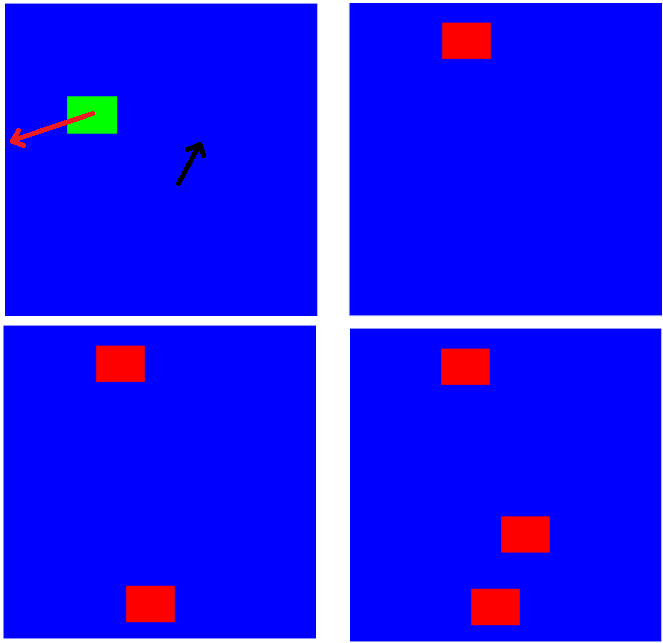}
		\captionsetup{justification=centering}
		\caption{Double rotation on a torus $\mathbb{T}$. The green rectangle \\indicates the smaller region. The blue parts represent $f^n(\mathbb{T})$ for $n=0,1,2,3$.}
		\label{fig:torus_rotation}
	\end{figure}

	\begin{figure}
	\centering
		\includegraphics[scale=1]{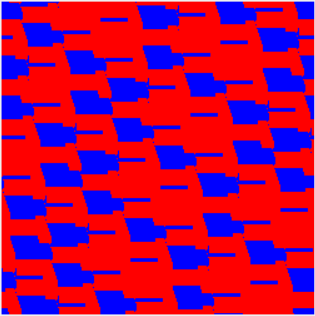}
		\captionsetup{justification=centering}
		\caption{The blue part is the attractor of the map in Figure \ref{fig:torus_rotation}.\\The attractor is attained after 500 iterates. The map is of finite type.}
		\label{fig:finite_attractor_torus}
	\end{figure}

	\begin{figure}
	\centering
		\includegraphics[scale=2.2]{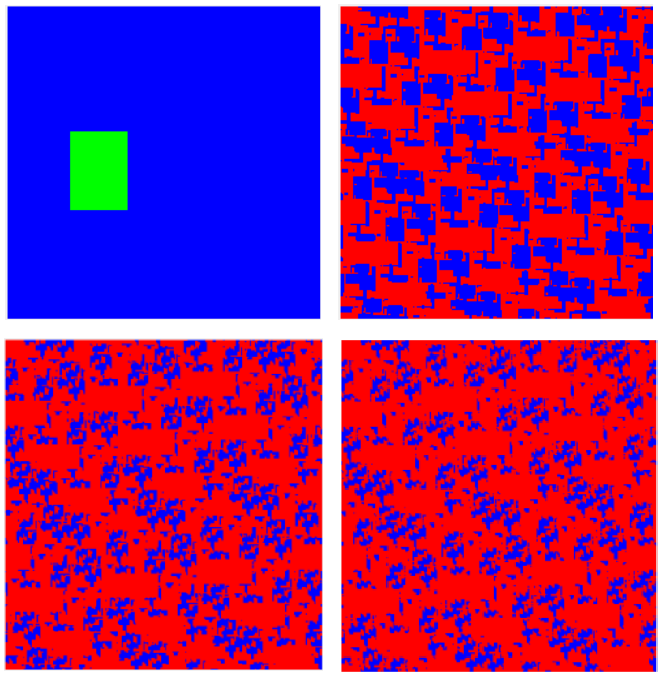}
		\captionsetup{justification=centering}
		\caption{Seemingly infinite type attractor.}
		\label{fig:infinite_attractor_torus}
	\end{figure}

\newpage	
\begin{center}\textbf{Acknowledgements} \end{center}
I would like to thank my advisor, Prof. Anton Gorodetski for the patient guidance, encouragement, and advice he provided during my research. I also want to thank UCI UROP for funding me during the summer of 2016 to finish this project.

%%%%%%%%%%%%%%%%%%%%%%%%%%%%%%%%%%%%%%%%%%%%%%%%%%%%%%%

%%%%%%%%%%%%%%%%%%%%%%%%%%%%%%%%%%%%%%%%%%%%%%%%%%%%%%%%
	
\end{document}